\documentclass[11pt,reqno]{amsart}
\usepackage[latin1]{inputenc} 
\usepackage{indentfirst}
\usepackage{epsfig} 
\usepackage{graphicx} 
\usepackage{multirow} 
\usepackage{verbatim} 
\usepackage{rotating} 
\usepackage{lettrine}
\usepackage{lscape}
\usepackage{ae}
\usepackage{float}
\usepackage[all]{xy}
\usepackage{icomma} 
\usepackage{varioref} 
\usepackage{layout}
\usepackage[Lenny]{fncychap} 
\usepackage{enumerate}
\usepackage{amsfonts, amsmath, amssymb, stmaryrd, latexsym, amsthm, dsfont}
\usepackage{dsfont}
\usepackage{wrapfig, color, epsfig}
\usepackage{xspace, setspace}
\usepackage{array}
\usepackage{longtable}
\usepackage{geometry}
\usepackage[frenchb,english]{babel}
\begin{document}
\newtheorem{The}{Theorem}[section]

\newcommand{\K}{\mathbb{Q}(\sqrt{d},i)}
\newcommand{\q}{\mathbb{Q}(\sqrt{d})/\mathbb{Q}}
\newcommand{\qq}{\mathbb{Q}(\sqrt{d})}
\newcommand{\pro}{\prod_{p/\Delta_K}e(p)}
\newtheorem{lem}[The]{Lemma}
\newtheorem{propo}[The]{Proposition}
\newtheorem{coro}[The]{Corollary}
\newtheorem{proprs}[The]{properties}
\newtheorem*{mainthem}{Theorem}
\theoremstyle{remark}
\newtheorem{rema}[The]{\bf Remark}
\newtheorem{remas}[The]{\bf Remarks}
\newtheorem{exam}[The]{\textbf{Numerical Example}}
\newtheorem{exams}[The]{\textbf{Numerical Examples}}
\newtheorem{df}[The]{definition}
\newtheorem{dfs}[The]{definitions}
\def\NN{\mathds{N}}
\def\RR{\mathbb{R}}
\def\HH{I\!\! H}
\def\QQ{\mathbb{Q}}
\def\CC{\mathbb{C}}
\def\ZZ{\mathbb{Z}}
\def\OO{\mathcal{O}}
\def\kk{\mathds{k}}
\def\KK{\mathbb{K}}
\def\FF{\mathbb{F}}
\def\ho{\mathcal{H}_0^{\frac{h(d)}{2}}}
\def\LL{\mathbb{L}}
\def\o{\mathcal{H}_0}
\def\h{\mathcal{H}_1}
\def\hh{\mathcal{H}_2}
\def\hhh{\mathcal{H}_3}
\def\hhhh{\mathcal{H}_4}
\def\k{\mathds{k}^{(*)}}
\def\G{\mathds{k}^{(*)}}
\def\l{\mathds{L}}
\def\L{\kk_2^{(2)}}
\def\kkk{k^{(*)}}
\def\Q{\mathcal{Q}}
\def\D{\kk_2^{(1)}}
\title[Capitulation in the absolutely abelian  extensions...]{Capitulation in the absolutely abelian  extensions of some number fields II}
\author[A. Azizi]{Abdelmalek Azizi}
\address{Abdelmalek Azizi: Mohammed First University, Mathematics Department, Sciences Faculty, Oujda, Morocco }
\email{abdelmalekazizi@yahoo.fr}
\author[A. Zekhnini]{Abdelkader Zekhnini}
\address{Abdelkader Zekhnini: Mohammed First University, Mathematics Department, Pluridisciplinary faculty, Nador, Morocco}
\email{zekha1@yahoo.fr}
\author[M. Taous]{Mohammed Taous}
\address{Mohammed Taous: Moulay Ismail University, Mathematics Department, Sciences and Techniques Faculty, Errachidia, Morocco.}
\email{taousm@hotmail.com}

\maketitle
\selectlanguage{english}
\begin{abstract}
We study the capitulation of $2$-ideal classes of an infinite
family of imaginary  biquadratic number fields consisting of
 fields $\kk =\QQ(\sqrt{pq_1q_2}, i)$, where $i=\sqrt{-1}$ and $q_1\equiv q_2\equiv-p\equiv-1 \pmod 4$
are different primes. For each of the three quadratic extensions $\KK/\kk$
inside the absolute genus field $\kk^{(*)}$ of $\kk$, we compute the capitulation
kernel of $\KK/\kk$. Then we deduce that each strongly ambiguous class of $\kk/\QQ(i)$
capitulates already in $\kk^{(*)}$.
\end{abstract}
\section{\bf Introduction and Notations}\label{sec1}
Let $k$ be an algebraic number field and let $\mathbf{C}l_2(k)$ denote its 2-class group, that is the 2-Sylow
subgroup of the ideal class group, $\mathbf{C}l(k)$,  of $k$.
We denote by  $k^{(*)}$  the absolute genus field of $k$, that is  the maximal abelian unramified extension of
$k$ obtained by composing  $k$ and an abelian extension over $\QQ$.

Suppose $F$ is a finite extension of $k$,
then we say that an ideal class of $k$ capitulates in $F$ if it is in the kernel of the
homomorphism
$$J_F: \mathbf{C}l(k) \longrightarrow \mathbf{C}l(F)$$ induced by extension of ideals from $k$
to $F$. An
important problem in Number Theory is to explicitly determine the kernel of  $J_F$, which is
usually called the capitulation kernel.

 If $F$ is
the  relative genus  field of a cyclic extension $K/k$, which we denote  by $(K/k)^*$ and that is  the maximal unramified extension  of  $K$ which is obtained by composing  $K$ and an abelian extension over $k$, F. Terada  states in  \cite{FT-71} that all the ambiguous ideal classes of $K/k$, which are classes of $K$ fixed under any element of $\mathrm{G}al(K/k)$,  capitulate in
 $(K/k)^*$.

  If $F$ is the  absolute genus  field of an abelian extension $K/\QQ$, then H. Furuya confirms in \cite{Fu-77} that every   strongly ambiguous  class of $K/\QQ$, that is an ambiguous ideal class containing at least one ideal invariant under  $\mathrm{G}al(K/\QQ)$, capitulates in  $F$.

   In this paper, we construct a family of number fields $k$ for which  all the strongly ambiguous  classes of $k/\QQ(i)$ capitulate in  $k^{(*)}\subset (k/\QQ(i))^*$.\par
 Let $\kk=\QQ(\sqrt{d}, i)$ and $\mathds{K}$ be an unramified quadratic extension of $\kk$ that is abelian  over $\QQ$. Denote by $\mathrm{A}m_s(\kk/\QQ(i))$
 the group of the strongly ambiguous classes of  $\kk/\QQ(i)$. In \cite{AZT14-1}, we studied the capitulation problem in the absolutely abelian extensions of $\kk$ for  $d=2pq$ and $p\equiv q\equiv1\pmod4$ are different primes, and in  \cite{AZT14-3}, we  dealt with  the same problem assuming $p\equiv -q\equiv1\pmod4$. In  \cite{AZT14-2, AZTM, AZT-15} and under the assumption $\mathbf{C}l_2(\kk)\simeq (2, 2, 2)$, we studied the capitulation problem of the $2$-ideal classes of $\kk$ in its fourteen unramified extensions, within the first Hilbert $2$-class field of $\kk$, and we gave the abelian type invariants of the $2$-class groups of these fourteen fields. Additionally we determined  the structure of the metabelian Galois group $G=\mathrm{Gal}(\L/\kk)$ of the second Hilbert 2-class field $\L$ of $\kk$.

 Let $q_1\equiv q_2\equiv-p\equiv-1 \pmod 4$ be different primes and  $d=pq_1q_2$. It is the purpose of the present article
to pursue this research project. We will  compute  the capitulation
kernel of $\KK/\kk$ and  we will  deduce  that   $\mathrm{A}m_s(\kk/\QQ(i))\subseteq \ker J_{\k}$. As
an application we will determine these kernels when  $\mathbf{C}l_2(\mathds{k})$ is of type $(2, 2, 2)$.

Let  $k$ be a number field, during this paper, we adopt the following notations:
 \begin{itemize}
   \item $\kappa_{K}$: the capitulation kernel of an unramified extension $K/k$.
   \item $\mathcal{O}_k$: the ring of integers of $k$.
   \item $E_k$: the unit group of $\mathcal{O}_k$.
   \item  $W_k$: the group of roots of unity contained in $k$.
   \item $k^+$: the maximal real subfield of $k$, if it is a CM-field.
   \item $Q_k=[E_k:W_kE_{k^+}]$ is Hasse's unit index, if $k$ is a CM-field.
   \item $q(k/\QQ)=[E_k:\prod_{i=1}^{s} E_{k_i}]$ is the unit index of $k$, if $k$ is multiquadratic, where $k_1, \dots, k_s$  are the  quadratic subfields of $k$.
  \item $k^{(*)}$: the absolute genus field of $k$.
  \item $\mathbf{C}l_2(k)$: the 2-class group of $k$.
  \item $i=\sqrt{-1}$.
  \item $\epsilon_m$: the fundamental unit of $\QQ(\sqrt m)$, if $m>1$ is a square-free integer,  that is a generator (modulo the roots of unity) for the unit group of the ring of integers of $\QQ(\sqrt m)$.
  \item $N(a)$: denotes the absolute norm of a number $a$, i.e., $N_{k/\QQ}(a)$ with $a\in k$ .
  \item $x\pm y$ means $x+y$ or $x-y$ for some numbers $x$ and $y$.
 \end{itemize}
\section{\bf Preliminary results}
Let us  first collect some results that will be useful in what follows.\par
Let $k_j$, $1\leq j\leq 3$, be the three real quadratic subfields of a biquadratic  real number field $K_0$
and $\epsilon_j>1$ be the fundamental unit of $k_j$. Since $$\alpha^2N_{K_0/{\bf Q}}(\alpha) =\prod_{j=1}^3N_{K_0/k_j}(\alpha)$$ for any $\alpha\in K_0$,
the square of any unit of $K_0$ is in the group generated by the $\epsilon_j$'s, $ 1\leq j\leq 3$.
Hence, to determine a fundamental system of units of $K_0$ it suffices to determine which of the units in
$B:=\{\epsilon_1, \epsilon_2, \epsilon_3, \epsilon_1\epsilon_2, \epsilon_1\epsilon_3, \epsilon_2\epsilon_3, \epsilon_1\epsilon_2\epsilon_3\}$
are  squares in $K_0$ (for details see \cite{Wa-66} or \cite{Lo-00}).
\begin{lem}[\cite{Wa-66}]\label{2}
A fundamental system of units of $K_0$ consists of three positive units chosen among $$B':=B\cup\{\sqrt \eta\ |\  \eta\in B \text{ and }\sqrt \eta\in K_0\}.$$
\end{lem}
\begin{lem}[\cite{Lo-00}]\label{4}
The units   $\epsilon\in B$ that can be squares in $K_0$ are as follows:
\begin{enumerate}[\rm1.]
\item $\epsilon=\epsilon_j$ and $N(\epsilon_j)=1$ with $1\leq j\leq 3$,
\item $\epsilon=\epsilon_j\epsilon_l$ and $N(\epsilon_j)=N(\epsilon_l)=1$ with $1\leq j\neq l\leq 3$,
\item $\epsilon=\epsilon_1\epsilon_2\epsilon_3$ and $N(\epsilon_1)=N(\epsilon_2)=N(\epsilon_3)$.
\end{enumerate}
\end{lem}
Put  $K=K_0(i)$,  then  to determine a  fundamental system of units  of $K$,  we will use the following result   that the second author has deduced from a theorem of Hasse \cite[\S 21, Satz 15 ]{Ha-52}.
\begin{lem}\cite[p.18]{Az-99}.\label{3}
Let $n\geq2$ be an integer   and $\xi_n$ a $2^n$-th primitive root of unity, then
$$\begin{array}{lllr}\xi_n = \dfrac{1}{2}(\mu_n + \lambda_ni), &\hbox{ where }&\mu_n =\sqrt{2 +\mu_{n-1}},&\lambda_n =\sqrt{2 -\mu_{n-1}}, \\
                               &    & \mu_2=0, \lambda_2=2 &\hbox{ and\quad } \mu_3=\lambda_3=\sqrt 2.
                                              \end{array}
                                            $$
Let $n_0$ be the greatest integer such that  $\xi_{n_0}$ is contained in $K$, $\{\epsilon'_1, \epsilon'_2, \epsilon'_3\}$ a  fundamental system of units  of $K_0$ and $\epsilon$ a unit of  $K_0$ such that $(2 + \mu_{n_0})\epsilon$ is a square in  $K_0$ $($if it exists$)$. Then a  fundamental system of units  of  $K$ is one of the following systems:
\begin{enumerate}[\rm1.]
\item $\{\epsilon'_1, \epsilon'_2, \epsilon'_3\}$ if $\epsilon$ does not exist,
\item $\{\epsilon'_1, \epsilon'_2, \sqrt{\xi_{n_0}\epsilon}\}$ if $\epsilon$
exists;  in this case $\epsilon = {\epsilon'_1}^{i_1}{
\epsilon'_2}^{i_2}\epsilon'_3$, where $i_1$, $i_2\in \{0, 1\}$ $($up to a permutation$)$.
\end{enumerate}
\end{lem}
\begin{lem}[\cite{Az-00}, {Lemma 5}]\label{1:046}
Let $d>1$ be a square-free integer and $\epsilon_d=x+y\sqrt d$,
where $x$, $y$ are  integers or semi-integers. If $N(\epsilon_d)=1$, then $2(x+1)$, $2(x-1)$, $2d(x+1)$ and
$2d(x-1)$ are not squares in  $\QQ$.
\end{lem}
\begin{lem}[\cite{Az-00}, {Lemma 6}]\label{1:048}
Let $q\equiv-1\pmod4$ be a prime and $\epsilon_q=x+y\sqrt q$ be the fundamental unit of  $\QQ(\sqrt q)$. Then $x$ is an even integer, $x\pm1$ is a square in  $\NN$ and $2\epsilon_q$ is a square in $\QQ(\sqrt q)$.
\end{lem}
\begin{lem}[\cite{Az-99}, {3.(1) p.19}]\label{6}
Let $d>2$ be a square-free integer and $k=\QQ(\sqrt d,i)$, put $\epsilon_d=x+y\sqrt d$.
\begin{enumerate}[\rm\indent1.]
  \item If $N(\epsilon_d)=-1$, then $\{\epsilon_d\}$ is a  fundamental system of units  of $k$.
  \item If $N(\epsilon_d)=1$, then $\{\sqrt{i\epsilon_d}\}$ is a fundamental system of units  of $k$ if
   and only if $x\pm1$ is a square in $\NN$ i.e. $2\epsilon_d$ is a square in $\QQ(\sqrt d)$. Else $\{\epsilon_d\}$ is  a  fundamental system of units  of $k$.
\end{enumerate}
\noindent This result is also in \cite{Kub-56}.
\end{lem}
 \begin{lem}[\cite{ZAT-15}]\label{3:105}
 Let $d\equiv1\pmod4$ be a positive square free integer and   $\epsilon_d=x+y\sqrt d$ be the fundamental unit of  $\QQ(\sqrt d)$. Assume   $N(\epsilon_d)=1$, then
 \begin{enumerate}[\rm\indent1.]
   \item $x+1$ and $x-1$ are not squares in  $\NN$ i.e. $2\epsilon_{d}$ is not a square in  $\QQ(\sqrt{d})$.
   \item For all prime  $p$ dividing   $d$, $p(x+1)$ and $p(x-1)$ are not squares in  $\NN$.
 \end{enumerate}
 \end{lem}
\section{\textbf{ Fundamental system of units  of some CM-fields}}
As  $\kk=\QQ(\sqrt{pq_1q_2}, i)$, so $\kk$ admits three unramified quadratic extensions that are abelian over $\QQ$, which are $\KK_1=\kk(\sqrt{p})=\QQ(\sqrt{p}, \sqrt{q_1q_2}, i)$,  $\KK_2=\kk(\sqrt{q_1})=\QQ(\sqrt{q_1}, \sqrt{q_2p}, i)$ and   $\KK_3=\kk(\sqrt{q_2})=\QQ(\sqrt{q_2}, \sqrt{q_1p}, i)$. Put $\epsilon_{pq_1q_2}=x+y\sqrt{pq_1q_2}$. In what follows, we determine the  fundamental system of units of $\KK_j$, $1\leq j\leq3$.
\subsection{\textbf{ Fundamental system of units  of the field  $\KK_1$}}~\\
We begin by determining the systems of fundamental  units of $\KK_1^+$ and $\KK_1$.
\begin{propo}\label{27}
Keep the previous notations. Then  $Q_{\KK_1}=1$ and
\begin{enumerate}[\rm1.]
\item If $2p(x\pm1)$ is a square in $\NN$,  then $\left\{\epsilon_{p},  \epsilon_{q_1q_2}, \sqrt{\epsilon_{pq_1q_2}}\right\}$ is a
  fundamental system of units of both of $\KK_1^+$ and $\KK_1$.
\item Else  $\left\{\epsilon_{p},  \epsilon_{q_1q_2}, \sqrt{\epsilon_{q_1q_2}\epsilon_{pq_1q_2}}\right\}$ is a  fundamental system of units of both of $\KK_1^+$ and $\KK_1$.
\end{enumerate}
\end{propo}
\begin{proof}
As   $N(\epsilon_{p})=-1$, then by Lemma \ref{4} only  $\epsilon_{q_1q_2}$, $\epsilon_{pq_1q_2}$ and $\epsilon_{q_1q_2}\epsilon_{pq_1q_2}$ can be squares in $\KK_1^+$.

Put $\epsilon_{q_1q_2}=a+b\sqrt{q_1q_2}$, then $a^2-1=b^2q_1q_2$. Hence by Lemmas  \ref{1:046} and \ref{3:105} we get that only the number   $2q_1(a\pm1)$ (i.e. $2q_2(a\pm1)$) is a square in $\NN$. So there exist    $b_1$ and $b_2$ in $\ZZ$ such that
 $$\left\{\begin{array}{rl}
    a\pm1 &=2b_1^2q_1\\
    a\mp1 &=2b_2^2q_2,
    \end{array}\right.$$
 therefore  $\sqrt{\epsilon_{q_1q_2}}=b_1\sqrt{q_1}+b_2\sqrt{q_2}$, which implies that$q_1\epsilon_{q_1q_2}$ and $q_2\epsilon_{q_1q_2}$ are squares in $\KK_1^+$ but $\epsilon_{q_1q_2}$ is not.

  Since $N(\epsilon_{pq_1q_2})=1$, then $x^2-1=y^2pq_1q_2$.  Hence Lemmas \ref{1:046} and \ref{3:105}  allowed us to distinguish the following cases:
 \begin{enumerate}[\rm a.]
 \item   If $2p(x\pm1)$ is a square in $\NN$, then $\sqrt{\epsilon_{pq_1q_2}}=y_1\sqrt{p}+y_2\sqrt{q_1q_2}$, hence   $\epsilon_{pq_1q_2}$ is a square in $\KK_1^+$.
 \item  If $2q_1(x\pm1)$ is a square in $\NN$,  then $\sqrt{\epsilon_{pq_1q_2}}=y_1\sqrt{q_1}+y_2\sqrt{pq_2}$, hence  $q_1\epsilon_{pq_1q_2}$ and $pq_2\epsilon_{pq_1q_2}$ are squares in $\KK_1^+$ but $\epsilon_{pq_1q_2}$ is not.
  \item   If $2q_2(x\pm1)$ is a square in $\NN$, then  $\sqrt{\epsilon_{pq_1q_2}}=y_1\sqrt{q_2}+y_2\sqrt{pq_1}$, hence   $q_2\epsilon_{pq_1q_2}$ and $pq_1\epsilon_{pq_1q_2}$ are  squares in  $\KK_1^+$  but $\epsilon_{pq_1q_2}$  is not.
 \end{enumerate}
 Consequently, we have
 \begin{enumerate}[\rm 1.]
 \item If  $2p(x\pm1)$ is a square in $\NN$, then $\epsilon_{pq_1q_2}$ is a square in $\KK_1^+$. Thus Lemmas \ref{2} and \ref{3} yield that $\left\{\epsilon_{p},  \epsilon_{q_1q_2}, \sqrt{\epsilon_{pq_1q_2}}\right\}$ is a   fundamental system of units of both of $\KK_1^+$ and $\KK_1$.
\item    If $2q_1(x\pm1)$ or $2q_2(x\pm1)$ is a square in $\NN$, then  $q_1\epsilon_{pq_1q_2}$ or $q_2\epsilon_{pq_1q_2}$ is a square in $\KK_1^+$. As $q_1\epsilon_{q_1q_2}$ and $q_2\epsilon_{q_1q_2}$ are squares in $\KK_1^+$, so $\epsilon_{q_1q_2}\epsilon_{pq_1q_2}$ is a square in $\KK_1^+$. Thus Lemmas \ref{2} and \ref{3} yield that  $\left\{\epsilon_{p},  \epsilon_{q_1q_2}, \sqrt{\epsilon_{q_1q_2}\epsilon_{pq_1q_2}}\right\}$ is a  fundamental system of units of both of $\KK_1^+$ and $\KK_1$.
\end{enumerate}
\end{proof}
\subsection{ Fundamental system of units of the field  $\KK_2$}~\\
Let us now determine the  fundamental system of units's of  $\KK_2^+=\QQ(\sqrt{q_1}, \sqrt{pq_2})$ and $\KK_2=\QQ(\sqrt{q_1}, \sqrt{pq_2}, i)$.
\begin{propo}\label{31}
Keep the previous notations and put $\epsilon_{pq_2}=a+b\sqrt{pq_2}$. Then   $Q_{\KK_2}=2$. Moreover we have:
\begin{enumerate}[\upshape1.]
\item Assume $2q_1(x\pm1)$ is a square in $\NN$,  then
\begin{enumerate}[\upshape i.]
\item If $a\pm1$ is a square in $\NN$, then $\left\{\epsilon_{q_1},  \sqrt{\epsilon_{q_1}\epsilon_{pq_2}}, \sqrt{\epsilon_{pq_1q_2}}\right\}$ is a  fundamental system of units of  $\KK_2^+$, and that of  $\KK_2$ is $\left\{\sqrt{\epsilon_{q_1}\epsilon_{pq_2}}, \sqrt{\epsilon_{pq_1q_2}}, \sqrt{i\epsilon_{q_1}}\right\}$.
\item Else   $\left\{\epsilon_{q_1}, \epsilon_{pq_2},  \sqrt{\epsilon_{pq_1q_2}}\right\}$ is a  fundamental system of units of  $\KK_2^+$, and that of  $\KK_2$ is  $\left\{\epsilon_{pq_2}, \sqrt{\epsilon_{pq_1q_2}}, \sqrt{i\epsilon_{q_1}}\right\}$.
\end{enumerate}
\item Assume $2q_1(x\pm1)$ is not a square in $\NN$,  then
\begin{enumerate}[\upshape i.]
\item If $a\pm1$ is a square in $\NN$, then $\left\{\epsilon_{q_1},  \sqrt{\epsilon_{q_1}\epsilon_{pq_2}}, \epsilon_{pq_1q_2}\right\}$ is a  fundamental system of units of   $\KK_2^+$, and that of $\KK_2$  is $\left\{\sqrt{\epsilon_{q_1}\epsilon_{pq_2}},  \epsilon_{pq_1q_2}, \sqrt{i\epsilon_{q_1}}\right\}$.
\item If $p(a\pm1)$ is a square in $\NN$, then $\left\{\epsilon_{q_1}, \epsilon_{pq_2}, \sqrt{\epsilon_{q_1}\epsilon_{pq_2}\epsilon_{pq_1q_2}}\right\}$ is a  fundamental system of units of   $\KK_2^+$, and that of  $\KK_2$ is $\left\{\epsilon_{pq_2},  \sqrt{\epsilon_{q_1}\epsilon_{pq_2}\epsilon_{pq_1q_2}}, \sqrt{i\epsilon_{q_1}}\right\}$.
\item If $2p(a\pm1)$ is a square in $\NN$, then $\left\{\epsilon_{q_1}, \epsilon_{pq_2}, \sqrt{\epsilon_{pq_2}\epsilon_{pq_1q_2}}\right\}$ is a  fundamental system of units of   $\KK_2^+$, and that of $\KK_2$ is  $\left\{\epsilon_{pq_2},  \sqrt{\epsilon_{pq_2}\epsilon_{pq_1q_2}}, \sqrt{i\epsilon_{q_1}}\right\}$.
\end{enumerate}
\end{enumerate}
\end{propo}
\begin{proof}
By Lemma \ref{4} the units that can be squares in $\KK_2$ are:  $\epsilon_{q_1}$, $\epsilon_{pq_2}$, $\epsilon_{pq_1q_2}$, $\epsilon_{q_1}\epsilon_{pq_2}$, $\epsilon_{q_1}\epsilon_{pq_1q_2}$, $\epsilon_{pq_1}\epsilon_{pq_1q_2}$ and $\epsilon_{q_1}\epsilon_{pq_2}\epsilon_{pq_1q_2}$.

According to  Lemma \ref{1:048},    $2\epsilon_{q_1}$ is a square in $\KK_2^+$ but $\epsilon_{q_1}$  is not.

 Put $\epsilon_{pq_2}=a+b\sqrt{pq_2}$, then $a^2-1=b^2pq_2$. Hence   Lemma \ref{1:046} allowed us to distinguish the following cases:
 \begin{enumerate}[\rm a.]
\item If $a\pm1$ is a square in $\NN$, then there exist $b_1$ and $b_2$ in $\ZZ$ such that
 $$\left\{\begin{array}{rl}
    a\pm1 &=b_1^2,\\
    a\mp1 &=b_2^2pq_2,
    \end{array}\right.$$
  thus $\sqrt{2\epsilon_{pq_2}}=b_1+b_2\sqrt{pq_2}$. Therefore  $2\epsilon_{pq_2}$  a square in $\KK_1^+$ but $\epsilon_{pq_2}$ is not.
  \item If $p(a\pm1)$ is a square in $\NN$, then there exist $b_1$ and $b_2$ in $\ZZ$ such that
 $$\left\{\begin{array}{rl}
    a\pm1 &=b_1^2p,\\
    a\mp1 &=b_2^2q_2,
    \end{array}\right.$$
 thus $\sqrt{2\epsilon_{pq_2}}=b_1\sqrt{p}+b_2\sqrt{q_2}$. Therefore   $2p\epsilon_{pq_2}$ and $2q_2\epsilon_{q_1q_2}$ are squares in $\KK_2^+$ but $\epsilon_{pq_2}$ and $2\epsilon_{pq_2}$ are not.
   \item If $2p(a\pm1)$ is a square in $\NN$, then there exist $b_1$ and $b_2$ in $\ZZ$ such
 $$\left\{\begin{array}{rl}
    a\pm1 &=2b_1^2p,\\
    a\mp1 &=2b_2^2q_2,
    \end{array}\right.$$
 thus $\sqrt{\epsilon_{pq_2}}=b_1\sqrt{p}+b_2\sqrt{q_2}$. Therefore  $p\epsilon_{pq_2}$ and $q_2\epsilon_{pq_2}$ are  squares in $\KK_2^+$ but $\epsilon_{pq_2}$  is not.
 \end{enumerate}

   As $N(\epsilon_{pq_1q_2})=1$, then  $x^2-1=y^2pq_1q_2$; hence Lemmas \ref{1:046} and \ref{3:105}  allowed us to distinguish the following cases:
    \begin{enumerate}[\rm a'.]
\item If $2p(x\pm1)$ is a square in $\NN$, then $\sqrt{\epsilon_{pq_1q_2}}=y_1\sqrt{p}+y_2\sqrt{q_1q_2}$, thus  $p\epsilon_{pq_1q_2}$ and $q_1q_2\epsilon_{pq_1q_2}$ are  squares in $\KK_2^+$ but $\epsilon_{pq_1q_2}$ is not.
 \item If  $2q_1(x\pm1)$ is a square in $\NN$,  then $\sqrt{\epsilon_{pq_1q_2}}=y_1\sqrt{q_1}+y_2\sqrt{pq_2}$, thus   $\epsilon_{pq_1q_2}$ is a square in $\KK_2^+$.
 \item If  $2q_2(x\pm1)$ is a square in $\NN$,  then $\sqrt{\epsilon_{pq_1q_2}}=y_1\sqrt{q_2}+y_2\sqrt{pq_1}$, thus $q_2\epsilon_{pq_1q_2}$ and $pq_1\epsilon_{pq_1q_2}$ are squares in $\KK_1^+$ but $\epsilon_{pq_1q_2}$  is not.
 \end{enumerate}
Consequently, we have
\begin{enumerate}[\rm 1.]
\item  Assume $2q_1(x\pm1)$ is a square in $\NN$, then $\epsilon_{pq_1q_2}$ is a square in $\KK_2^+$.
\begin{enumerate}[\rm i.]
\item If $a\pm1$ is a square in $\NN$, then  $2\epsilon_{pq_2}$ is a square in $\KK_2^+$;  thus $\epsilon_{q_1}\epsilon_{pq_2}$ is a square in $\KK_2^+$, since  $2\epsilon_{q_1}$ is. Therefore, by Lemma  \ref{2} $\left\{\epsilon_{q_1},  \sqrt{\epsilon_{q_1}\epsilon_{pq_2}}, \sqrt{\epsilon_{pq_1q_2}}\right\}$ is a  fundamental system of units of  $\KK_2^+$, and according to Lemma \ref{3}  $\left\{\sqrt{\epsilon_{q_1}\epsilon_{pq_2}}, \sqrt{\epsilon_{pq_1q_2}}, \sqrt{i\epsilon_{q_1}}\right\}$   is a  fundamental system of units of   $\KK_2$.
 \item Else    $\epsilon_{pq_1q_2}$ will be a square in $\KK_2^+$; hence by Lemma  \ref{2} $\left\{\epsilon_{q_1},  \epsilon_{pq_2}, \sqrt{\epsilon_{pq_1q_2}}\right\}$ is a  fundamental system of units of  $\KK_2^+$,  and according to Lemma \ref{3} $\left\{\epsilon_{pq_2}, \sqrt{\epsilon_{pq_1q_2}}, \sqrt{i\epsilon_{q_1}}\right\}$   is a  fundamental system of units of   $\KK_2$.
 \end{enumerate}
\item Assume $2q_1(x\pm1)$ is not a square in $\NN$,  then $\epsilon_{pq_1q_2}$ is not a square in $\KK_2^+$.
\begin{enumerate}[\rm i.]
\item If  $a\pm1$   is a square in $\NN$, then $2\epsilon_{pq_2}$ is a square in $\KK_2^+$;  hence  $\epsilon_{q_1}\epsilon_{pq_2}$ is a square in $\KK_2^+$, since $2\epsilon_{q_1}$ is a square in $\NN$. Thus  by Lemma  \ref{2} $\left\{\epsilon_{q_1},   \sqrt{\epsilon_{q_1}\epsilon_{pq_2}}, \epsilon_{pq_1q_2}\right\}$ is a  fundamental system of units of   $\KK_2^+$,  and according to Lemma \ref{3}  $\left\{\sqrt{\epsilon_{q_1}\epsilon_{pq_2}}, \epsilon_{pq_1q_2}, \sqrt{i\epsilon_{q_1}}\right\}$ is a  fundamental system of units of   $\KK_2$.
\item If $p(a\pm1)$   is a square in $\NN$, then $2p\epsilon_{pq_2}$ and $2q_2\epsilon_{pq_2}$ are squares in $\KK_2^+$. On the other hand, we have $p\epsilon_{pq_1q_2}$ or $q_2\epsilon_{pq_1q_2}$ is a square in $\KK_1^+$, thus  $\epsilon_{q_1}\epsilon_{pq_2}\epsilon_{pq_1q_2}$ is a square in $\KK_2^+$, since $2\epsilon_{q_1}$ is a square in $\NN$. Therefore  by Lemma  \ref{2}  $\left\{\epsilon_{q_1},   \epsilon_{pq_2}, \sqrt{\epsilon_{q_1}\epsilon_{pq_2}\epsilon_{pq_1q_2}}\right\}$ is a  fundamental system of units of   $\KK_2^+$,  and according to Lemma \ref{3}\\ $\left\{\epsilon_{pq_2}, \sqrt{\epsilon_{q_1}\epsilon_{pq_2}\epsilon_{pq_1q_2}}, \sqrt{i\epsilon_{q_1}}\right\}$ is a  fundamental system of units of   $\KK_2$.
\item The last case is treated similarly.
\end{enumerate}

  \end{enumerate}
\end{proof}
\subsection{ Fundamental system of units of the field  $\KK_3$}~\\
Since $q_1$ and $q_2$ play symmetrical roles, then the  fundamental system of units's of  $\KK_3^+=\QQ(\sqrt{q_2}, \sqrt{pq_1})$ and $\KK_3=\QQ(\sqrt{q_2}, \sqrt{pq_1}, i)$ are easily deduced.
\begin{propo}\label{32}
Keep the previous notations and put  $\epsilon_{pq_1}=a+b\sqrt{pq_1}$. Then $Q_{\KK_3}=2$. Moreover we have.
\begin{enumerate}[\upshape1.)]
\item Assume $2q_2(x\pm1)$ is a square in $\NN$,  then
\begin{enumerate}[\upshape i.]
\item If $a\pm1$ is a square in $\NN$, then $\left\{\epsilon_{q_2},  \sqrt{\epsilon_{q_2}\epsilon_{pq_1}}, \sqrt{\epsilon_{pq_1q_2}}\right\}$ is a  fundamental system of units of  $\KK_3^+$, and that of  $\KK_3$ is $\left\{\sqrt{\epsilon_{q_2}\epsilon_{pq_1}}, \sqrt{\epsilon_{pq_1q_2}}, \sqrt{i\epsilon_{q_2}}\right\}$.
\item Else   $\left\{\epsilon_{q_2}, \epsilon_{pq_1},  \sqrt{\epsilon_{pq_1q_2}}\right\}$ is a  fundamental system of units of  $\KK_3^+$, and that of  $\KK_3$ is $\left\{\epsilon_{pq_1}, \sqrt{\epsilon_{pq_1q_2}}, \sqrt{i\epsilon_{q_2}}\right\}$.
\end{enumerate}
\item Assume $2q_2(x\pm1)$ is not a square in $\NN$,  then
\begin{enumerate}[\upshape i.]
\item If $a\pm1$ is a square in $\NN$, then $\left\{\epsilon_{q_2},  \sqrt{\epsilon_{q_2}\epsilon_{pq_1}}, \epsilon_{pq_1q_2}\right\}$ is a  fundamental system of units of   $\KK_3^+$, and that of $\KK_3$ is   $\left\{\sqrt{\epsilon_{q_2}\epsilon_{pq_1}},  \epsilon_{pq_1q_2}, \sqrt{i\epsilon_{q_2}}\right\}$.
\item If $p(a\pm1)$ is a square in $\NN$, then $\left\{\epsilon_{q_2}, \epsilon_{pq_1}, \sqrt{\epsilon_{q_2}\epsilon_{pq_1}\epsilon_{pq_1q_2}}\right\}$ is a  fundamental system of units of   $\KK_3^+$, and that of $\KK_2$ is   $\left\{\epsilon_{pq_1},  \sqrt{\epsilon_{q_2}\epsilon_{pq_1}\epsilon_{pq_1q_2}}, \sqrt{i\epsilon_{q_2}}\right\}$.
\item If $2p(a\pm1)$ is a square in $\NN$, then $\left\{\epsilon_{q_2}, \epsilon_{pq_1}, \sqrt{\epsilon_{pq_1}\epsilon_{pq_1q_2}}\right\}$ is a  fundamental system of units of   $\KK_3^+$, and that of $\KK_3$ is   $\left\{\epsilon_{pq_1},  \sqrt{\epsilon_{pq_1}\epsilon_{pq_1q_2}}, \sqrt{i\epsilon_{q_2}}\right\}$.
\end{enumerate}
\end{enumerate}
\end{propo}
\section{\textbf{The  ambiguous classes of $\kk/\QQ(i)$}}
 Let $F=\QQ(i)$ and $\kk=\QQ(\sqrt{pq_1q_2}, i)$. We denote by  $\mathrm{A}m(\kk/F)$  the group of the ambiguous classes of $\kk/F$ and by $\mathrm{A}m_s(\kk/F)$ the subgroup of $\mathrm{A}m(\kk/F)$ generated by  the strongly ambiguous classes.
As  $p\equiv 1\pmod4$, so there exist $e$ and $f$ in $\NN$ such that  $p=e^2+4f^2=\pi_1\pi_2$. Put  $\pi_1=e+2if$ and  $\pi_2=e-2if$. Let  $\mathcal{H}_j$ (resp. $\mathcal{Q}_j$)  be the prime ideal of  $\kk$ above $\pi_j$ (resp. $q_j$), where $j\in\{1, 2\}$. It is easy to see  that $\mathcal{H}_j^2=(\pi_j)$ and $\mathcal{Q}_j^2=(q_j)$. Therefore $[\mathcal{Q}_j]$ and  $[\mathcal{H}_j]$ are in $\mathrm{A}m_s(\kk/F)$, for all  $j\in\{1, 2\}$.  Keep the notation $\epsilon_{pq_1q_2}=x+y\sqrt{pq_1q_2}$. In this section,  we will determine  generators of $\mathrm{A}m_s(\kk/F)$ and $\mathrm{A}m(\kk/F)$.  Let us first prove the following result.
\begin{lem}\label{7}
Consider the prime ideals $\mathcal{H}_j$ and $\mathcal{Q}_j$ of $\kk$, $1\leq j\leq2$.
\begin{enumerate}[\rm\indent1.]
  \item If $2p(x\pm1)$ is  a square in $\NN$, then  $\left|\langle[\h], [\mathcal{Q}_1]\rangle\right|=4$.
     \item Else $\left|\langle[\h], [\hh]\rangle\right|=4$
\end{enumerate}
\end{lem}
\begin{proof}
Since $\mathcal{H}_j^2=(\pi_j)$, $1\leq j\leq2$, and since also  $\sqrt{e^2+(2f)^2}=\sqrt{p}\not\in\QQ(\sqrt{pq_1q_2})$, so,  according to \cite[Proposition 1]{AZT12-2},  $\mathcal{H}_j$ are not  principal  in $\kk$.\par
1. If $2p(x\pm1)$ is  a square in $\NN$, and since
$(\h\hh)^2=(p_1)$,  $\mathcal{Q}_j^2=(q_j)$ and $(\h\mathcal{Q}_j)^2=(q_j)$, hence by \cite[Proposition 2 and Remark 1]{AZT12-2},  $\h\hh$ is principal in $\kk$ and $\mathcal{Q}_j$, $\h\mathcal{Q}_j$ are not. Thus the result.

2. If $2p(x\pm1)$ is not a square in $\NN$, i.e. $2q_1(x\pm1)$ or $2q_2(x\pm1)$ is a square in $\NN$; then $\h\hh$ is not principal in $\kk$ and $\mathcal{Q}_1$ or $\mathcal{Q}_2$ is (by \cite[Proposition 2]{AZT12-2}). On the other hand, if $\mathcal{Q}_1$ (resp. $\mathcal{Q}_2$) is principal,  then $[\h\hh]=[\mathcal{Q}_2]$ (resp. $[\h\hh]=[\mathcal{Q}_1]$).
\end{proof}

 Determine now  generators of $ \mathrm{A}m_s(\kk/F)$ and $\mathrm{A}m(\kk/F)$. According to the ambiguous class number formula (\cite{Ch-33}),
the genus number, $[(\kk/F)^*:\kk]$, is given by:
 \begin{equation}\label{51}
|\mathrm{A}m(\kk/F)|=[(\kk/F)^*:\kk]=\frac{h(F)2^{t-1}}{[E_F: E_F\cap N_{\kk/F}(\kk^\times)]},
\end{equation}
where $h(F)$ is the class number of $F$ and $t$ is the number of finite and infinite primes of $F$ ramified in $\kk/F$. Moreover as the class number
of  $F$ is equal to $1$, so the formula \eqref{51} yields that
 \begin{equation}\label{56}|\mathrm{A}m(\kk/F)|=[(\kk/F)^*:\kk]=2^r,\end{equation}
 where $r=\text{rank}\mathbf{C}l_2(\mathds{k})=t-e-1$ and $2^e=[E_F: E_F\cap N_{\kk/F}(\kk^\times)]$ (see for example \cite{McPaRa-95}).
  The relation between  $|\mathrm{A}m(\kk/F)|$ and $|\mathrm{A}m_s(\kk/F)|$ is given by the following formula (see for example \cite{Lem-13}):
\begin{equation}\label{50}
\frac{|\mathrm{A}m(\kk/F)|}{|\mathrm{A}m_s(\kk/F)|}=[E_F\cap N_{\kk/F}(\kk^\times):N_{\kk/F}(E_\kk)].
\end{equation}
 To continue, we need  the following lemma.
\begin{lem}\label{57}
Let $p\equiv -q_1\equiv-q_2\equiv1\pmod4$ be different primes,  $F=\QQ(i)$ and $\kk=\QQ(\sqrt{pq_1q_2}, i)$.
\begin{enumerate}[\rm\indent1.]
  \item If $p\equiv1 \pmod8$, then $i$ is  a norm in $\kk/F$.
  \item If $p\equiv 5\pmod8$, then $i$ is not a norm in $\kk/F$.
\end{enumerate}
\begin{proof}
  We proceed as in Lemma 11 of \cite{AZT14-3}.
\end{proof}
\end{lem}
\begin{propo}\label{248}
Let $(\kk/F)^*$ denote the relative genus field of $\kk/F$. Then
\begin{enumerate}[\rm\indent1.]
\item
\begin{enumerate}[\rm i.]
\item If $p\equiv1\pmod8$, then $\k\varsubsetneq (\kk/F)^*$ and $[(\kk/F)^*:\k]=2$.
\item Else $\k= (\kk/F)^*$.
\end{enumerate}
\item Assume  $p\equiv1 \pmod8$.
\begin{enumerate}[\rm i.]
  \item If $2p(x\pm1)$ is a square in  $\NN$, then $\mathrm{Am}_s(\kk/\QQ(i))=\langle[\h], [\mathcal{Q}_1]\rangle.$
  \item Else, $\mathrm{Am}_s(\kk/\QQ(i))=\langle[\h], [\hh]\rangle.$
   \item  there  exist an unambiguous ideal
   $\mathcal{I}$ in $\kk/\QQ(i)$ of order  $2$ such that
   $$\mathrm{Am}(\kk/\QQ(i))=
    \left\{\begin{array}{ll}
     \langle[\h], [\mathcal{Q}_1], [\mathcal{I}]\rangle,& \text{ if }2p(x\pm1)\text{ is a square in }\NN,\\
     \langle[\h], [\hh], [\mathcal{I}]\rangle, & \text{ otherwise}.
     \end{array}\right.$$
\end{enumerate}
\item Assume  $p\equiv5\pmod8$, then
$$\mathrm{Am}(\kk/\QQ(i))=\mathrm{Am}_s(\kk/\QQ(i))=
    \left\{\begin{array}{ll}
     \langle[\h], [\mathcal{Q}_1]\rangle,& \text{ if }2p(x\pm1)\text{ is a square in }\NN,\\
     \langle[\h], [\hh]\rangle, & \text{ otherwise}.
     \end{array}\right.$$
\end{enumerate}
\end{propo}
\begin{proof}
1. As $\kk=\QQ(\sqrt{pq_1q_2}, i)$, so $[\k:\kk]=4$. Moreover, according to \cite[Proposition 2, p. 90]{McPaRa-95}, $r=\text{rank}\mathbf{C}l_2(\mathds{k})=3$ if $p\equiv1 \pmod8$ and $r=\text{rank}\mathbf{C}l_2(\mathds{k})=2$ if $p\equiv5 \pmod8$, so $[(\kk/F)^*:\kk]=4 \text{ or } 8$. Hence $[(\kk/F)^*:\k]=1 \text{ or } 2$, and the   results derived.\par
2. Note first that, by Lemma \ref{3:105}, $x+1$ and $x-1$ are never  squares in  $\NN$. Thus from Lemma \ref{6} we get $E_\kk=\langle i, \epsilon_{pq_1q_2}\rangle$.\\
  Assume  $p\equiv1 \pmod8$, hence $i$ is a norm in $\kk/\QQ(i)$ (Lemma \ref{57}), thus Formula \eqref{50} yields that
    \begin{align*}\dfrac{|\mathrm{Am}(\kk/\QQ(i))|}{|\mathrm{Am}_s(\kk/\QQ(i))|}=[E_{\QQ(i)}\cap N_{\kk/\QQ(i)}(\kk^{\times}):N_{\kk/\QQ(i)}(E_\kk)]
                                                                               =2 \end{align*}
since  $[E_{\QQ(i)}\cap
N_{\kk/\QQ(i)}(\kk^{\times}):N_{\kk/\QQ(i)}(E_\kk)]=[< i >: < -1>]=2$.\\
\indent On the other hand, as  $p\equiv1 \pmod8$, we have just shown that   $r=3$. Therefore     $|\mathrm{Am}(\kk/\QQ(i))|=2^4$ and thus $|\mathrm{Am}_s(\kk/\QQ(i))|=4$\\
\indent i. If $2p(x\pm1)$ is a square in $\NN$ which is equivalent to $2q_1q_2(x\pm1)$ is a square in $\NN$, then $\mathrm{Am}(\kk/\QQ(i))=2\mathrm{Am}_s(\kk/\QQ(i))$, hence by Lemma \ref{7} we get  $$\mathrm{Am}_s(\kk/\QQ(i))=\langle[\h], [\mathcal{Q}_1]\rangle.$$
\indent ii.  If $2q_1(x\pm1)$ or $2q_2(x\pm1)$ is a square in $\NN$,  then    Lemma \ref{7} yields that $$\mathrm{Am}_s(\kk/\QQ(i))=\langle[\h], [\hh]\rangle.$$
\indent Consequently, in the two cases  there exists an unambiguous  ideal $\mathcal{I}$ in $\kk/F$ of order $2$ such that
    $$\mathrm{Am}(\kk/\QQ(i))=
    \left\{\begin{array}{ll}
     \langle[\h], [\mathcal{Q}_1], [\mathcal{I}]\rangle,& \text{ if }2p(x\pm1)\text{ is a square in }\NN,\\
     \langle[\h], [\hh], [\mathcal{I}]\rangle, & \text{ else}.
     \end{array}\right.$$
    By Chebotarev theorem,  $\mathcal{I}$ can  always be chosen  as a prime ideal of $\kk$ above a prime $\ell$ in $\QQ$, which splits completely in $\kk$.

3. Assume   $p\equiv5\pmod8$, hence $i$ is not a norm in $\kk/\QQ(i)$ (Lemma \ref{57}). Proceeding similarly as in 2., we get
  $$\mathrm{Am}(\kk/\QQ(i))=\mathrm{Am}_s(\kk/\QQ(i))=
    \left\{\begin{array}{ll}
     \langle[\h], [\mathcal{Q}_1]\rangle,& \text{ if }2p(x\pm1)\text{ is a square in }\NN,\\
     \langle[\h], [\hh]\rangle, & \text{ else}.
     \end{array}\right.$$
      This completes the proof.
\end{proof}
\section{\bf{Capitulation}}
Let  $p$,  $q_1$ and $q_2$ be primes satisfying  $p\equiv-q_1\equiv-q_2\equiv1 \pmod 4$. Set  $\kk=\QQ(\sqrt{pq_1q_2}, i)$ and denote by $\G$ the genus field of   $\kk$, then $\k=\QQ(\sqrt p, \sqrt{q_1}, \sqrt{q_2}, i)$. The unramified quadratic extensions of $\kk$, abelian over $\QQ$,  are  $\KK_1=\kk(\sqrt{p})=\QQ(\sqrt{p},\sqrt{q_1q_2}, i)$, $\KK_2=\kk(\sqrt{q_1})=\QQ(\sqrt{q_1},\sqrt{pq_2},i)$ and  $\KK_3=\kk(\sqrt{q_2})=\QQ(\sqrt{q_2}, \sqrt{pq_1}, i)$.  Keep the notations  $\epsilon_{pq_1q_2}=x+y\sqrt{pq_1q_2}$ denoting the fundamental unit of   $\QQ(\sqrt{pq_1q_2})$ and  $p=e^2+4f^2=\pi_1\pi_2$, where $\pi_1=e+2if$, $\pi_2=e-2if$. Let $Q_{\kk}$ be the unit index of  $\kk$, and $\mathcal{H}_j$ be the ideal of  $\kk$ lies above $\pi_j$. Denote also by $\mathcal{Q}_j$ the prime ideal of $\kk$ above $q_j$ , $j=1, 2$.

In this section,  we will determine the classes of  $\mathbf{C}l_2(\kk)$, the  $2$-class group of $\kk$, that capitulate in  $\KK_j$, for all $j\in\{1, 2, 3\}$. For this we need the following theorem.
\begin{The}[\cite{HS82}]\label{1}
 Let $K/k$ be a cyclic extension of prime degree, then  the number of classes that capitulate in $K/k$ is:
 $[K:k][E_k:N_{K/k}(E_K)],$
 where $E_k$ and $E_K$ are the  unit groups of $k$ and $K$ respectively.
\end{The}
\subsection{The number of classes  capitulating in each $\KK_j$}~\\
Recall that $\kappa_{\KK_j}$ denotes  the capitulation kernel of the unramified extension $\KK_j/\kk$.
\begin{The}\label{226}
Let $\KK_j$, $1\leq j\leq3$, be the three unramified quadratic extensions of $\kk$ defined above. Then
\begin{enumerate}[\rm1.]
\item  $|\kappa_{\KK_1}|=4$.
\item Let $\epsilon_{pq_2}=a+b\sqrt{pq_2}$, then
\begin{enumerate}[\rm i.]
\item If  $a\pm1$ is a square in $\NN$ and $2q_1(x+1)$, $2q_1(x-1)$ are not, then $|\kappa_{\KK_2}|=4$.
\item In the other cases  $|\kappa_{\KK_2}|=2$.
\end{enumerate}
\item Let  $\epsilon_{pq_1}=a+b\sqrt{pq_1}$, then
\begin{enumerate}[\rm i.]
\item If  $a\pm1$ is a square in $\NN$  and    $2q_2(x+1)$, $2q_2(x-1)$ are not, then $|\kappa_{\KK_3}|=4$.
\item In the other cases  $|\kappa_{\KK_3}|=2$.
\end{enumerate}
\end{enumerate}
\end{The}
\begin{proof}
 Note first that, according to   Lemma  \ref{3:105},  $x+1$ and $x-1$ are never squares in $\NN$, hence by  Lemma \ref{6}, $E_{\kk}=\langle i, \epsilon_{pq_1q_2}\rangle$.
 \begin{enumerate}[\rm1.]
 \item By   Proposition \ref{27} we have
  $E_{\KK_1}=\langle  i, \epsilon_{p}, \epsilon_{q_1q_2}, \sqrt{\epsilon_{pq_1q_2}}\rangle$ or
   $E_{\KK_1}=\langle  i, \epsilon_{p}, \epsilon_{q_1q_2},\\ \sqrt{\epsilon_{q_1q_2}\epsilon_{pq_1q_2}}\rangle$, hence $N_{\KK_1/\kk}(E_{\KK_1})=\langle -1, \epsilon_{pq_1q_2}\rangle$. Thus   $[E_{\kk}:N_{\KK_1/\kk}(E_{\KK_1})]=2$. Therefore  Theorem \ref{1} implies that  $|\kappa_{\KK_1}|=4$.

\item  i.  If  $a\pm1$ is a square in $\NN$ and $2q_1(x+1)$, $2q_1(x-1)$ are not, then    Proposition \ref{31}(2)(i) yields that $N_{\KK_2/\kk}(E_{\KK_2})=\langle i, \epsilon_{pq_1q_2}^2\rangle$, hence $[E_{\kk}:N_{\KK_2/\kk}(E_{\KK_2})]=2$. Thus  Theorem  \ref{1} implies that  $|\kappa_{\KK_2}|=4$.\\
   ii.  The other cases are grouped together in  Proposition \ref{31} (assertions 1, 2), then  $N_{\KK_2/\kk}(E_{\KK_2})=\langle i, \epsilon_{pq_1q_2}\rangle$. Thus  $[E_{\kk}:N_{\KK_2/\kk}(E_{\KK_2})]=1$,  and Theorem  \ref{1} implies that   $|\kappa_{\KK_2}|=2$.
\item This point is similarly treated.
\end{enumerate}
\end{proof}
\subsection{Capitulation in $\KK_1$}
\begin{The}\label{233}
Let   $p$,  $q_1$ and $q_2$ be different  primes such that  $p\equiv-q_1\equiv-q_2\equiv1 \pmod 4$. Put  $\kk=\QQ(\sqrt{pq_1q_2}, i)$,
$\KK_1=\QQ(\sqrt{p}, \sqrt{q_1q_2}, i)$ and $\epsilon_{pq_1q_2}=x+y\sqrt{pq_1q_2}$, then
\begin{enumerate}[\rm1.]
\item If $2p(x\pm1)$ is a square in  $\NN$, then $\kappa_{\KK_1}=\langle[\h], [\Q_1]\rangle$.
\item Else,  $\kappa_{\KK_1}=\langle[\h], [\hh]\rangle$.
\end{enumerate}
\end{The}
\begin{proof}
We have already shown, in Lemma \ref{7}, that $\h$, $\hh$,  $\Q_j$ and $\mathcal{H}_k\Q_j$, $j, k=1,\ 2$, are not principal  in  $\kk$.
On the other hand, by  Proposition 6.3 of \cite{AZT-16} $\h$ and $\hh$ capitulate in $\KK_1$.

1. If $2p(x\pm1)$ is a square in  $\NN$, then  by  \cite[Proposition 2]{AZT12-2}  $\h\hh$ is principal in $\kk$, i.e. $[\h]=[\hh]$. The proof of  the Proposition \ref{27}, allows us to conclude that $q_1\epsilon_{q_1q_2}$ and $q_2\epsilon_{q_1q_2}$ are squares in  $\KK_1$; hence  there exists  $\gamma\in\KK_1$ such that $\Q_1^2=(\gamma^2)$. Thus $\Q_1=(\gamma)$, so the result.\\
2. If $2p(x+1)$ and $2p(x-1)$ are not squares in $\NN$, then $\h\hh$ is not principal in $\kk$; which yields the result.
\end{proof}
\begin{exams}~\\
$1.$ The case where $2p(x\pm1)$ is a square in $\NN$.
\footnotesize
\begin{longtable}{| c | c | c | c | c | c | c | }
\hline
 $d = p.q_1.q_2$  &  $2p(x+1)$ & $2p(x-1)$ & $\h\hh$ in $\kk$ & $\Q_1$ in $\kk$  & $\h$ & $\Q_1$ \\
\hline
\endfirsthead
\hline
 $d = p.q_1.q_2$  &  $2p(x+1)$ & $2p(x-1)$ & $\h\hh$ in $\kk$ & $\Q_1$ in $\kk$  & $\h$ & $\Q_1$ \\
\hline
\endhead
$105=5.3.7$ & $420$ & $400=20^2$ & $[0, 0]~$ & $[2, 0]~$ & $[0, 0]~$ & $[0, 0]~$ \\ \hline
$345=5.23.3$ & $67620$ & $67600=260^2$ & $[0, 0]~$ & $[2, 0]~$ & $[0, 0]~$ & $[0, 0]~$ \\ \hline
$357=17.3.7$ & $357$ & $289=17^2$ & $[0, 0, 0]~$ & $[1, 1, 0]~$ & $[0, 0]~$ & $[0, 0]~$ \\ \hline
$561=17.11.3$ & $17774724$ & $17774656=4216^2$ & $[0, 0, 0]~$ & $[0, 0, 1]~$ & $[0, 0, 0]~$ & $[0, 0, 0]~$ \\ \hline
$645=5.3.43$ & $645$ & $625=25^2$ & $[0, 0]~$ & $[4, 0]~$ & $[0, 0]~$ & $[0, 0]~$ \\ \hline
$705=5.47.3$ & $2371620$ & $2371600=1540^2$ & $[0, 0]~$ & $[6, 0]~$ & $[0, 0]~$ & $[0, 0]~$ \\ \hline
$805=5.7.23$ & $7245$ & $7225=85^2$ & $[0, 0]~$ & $[4, 0]~$ & $[0, 0]~$ & $[0, 0]~$ \\ \hline
\end{longtable}
\normalsize
$2.$ The case where  $2p(x+1)$ and $2p(x-1)$ are not squares in $\NN$.
\small
\begin{longtable}{| c | c | c | c | c | c | }
\hline
 $d = p.q_1.q_2$  &  $2p(x+1)$ & $2p(x-1)$ & $\h\hh$ in $\kk$  & $\h$ & $\hh$ \\
\hline
\endfirsthead
\hline
 $d = p.q_1.q_2$  &  $2p(x+1)$ & $2p(x-1)$ & $\h\hh$ in $\kk$ & $\h$ & $\hh$ \\
\hline
\endhead
$165=5.3.11$ & $75$ & $55$ & $[2, 0]~$ & $[0, 0]~$ & $[0, 0]~$ \\ \hline
$273=13.7.3$ & $18928$ & $18876$ & $[2, 0]~$ & $[0, 0]~$ & $[0, 0]~$ \\ \hline
$285=5.3.19$ & $95$ & $75$ & $[4, 0]~$ & $[0, 0]~$ & $[0, 0]~$ \\ \hline
$429=13.11.3$ & $1911$ & $1859$ & $[4, 0]~$ & $[0, 0]~$ & $[0, 0]~$ \\ \hline
$465=5.3.31$ & $158720$ & $158700$ & $[4, 0]~$ & $[0, 0]~$ & $[0, 0]~$ \\ \hline
$609=29.7.3$ & $35130368$ & $35130252$ & $[4, 0]~$ & $[0, 0]~$ & $[0, 0]~$ \\ \hline
$665=5.7.19$ & $137200$ & $137180$ & $[6, 0]~$ & $[0, 0]~$ & $[0, 0]~$ \\ \hline
$741=13.19.3$ & $3211$ & $3159$ & $[6, 0]~$ & $[0, 0]~$ & $[0, 0]~$ \\ \hline
$1533=73.3.7$ & $37303$ & $37011$ & $[3, 1, 0]~$ & $[0, 0, 0]~$ & $[0, 0, 0]~$ \\ \hline
\end{longtable}
\normalsize
\end{exams}
\subsection{Capitulation in $\KK_2$}~\\
Let   $p$,  $q_1$ and $q_2$ be different  primes such that  $p\equiv-q_1\equiv-q_2\equiv1 \pmod 4$. Put  $\kk=\QQ(\sqrt{pq_1q_2}, i)$,
$\KK_2=\QQ(\sqrt{q_1},\sqrt{pq_2}, i)$ and $\epsilon_{pq_2}=a+b\sqrt{pq_2}$.
\begin{lem}
If $a\pm1$ is a square in $\NN$, then $p\equiv1\pmod8$.
\end{lem}
\begin{proof}
If $a\pm1$ is a square in $\NN$, then
 $ \left\{\begin{array}{ll}
a\pm 1 &=y_1^2,\\
a\mp 1 &= pq_2y_2^2.
\end{array}\right.$ \\
Hence  $1=\left(\frac{a\pm1}{p}\right)=\left(\frac{a\mp1\pm2}{p}\right)=\left(\frac{2}{p}\right)$.
\end{proof}
\noindent Therefore, if we suppose that $a\pm1$ is a square in $\NN$, then from Proposition \ref{248} we get:

\begin{enumerate}[\rm i.]
  \item If $2p(x\pm1)$ is a square in  $\NN$, then $\mathrm{Am}_s(\kk/\QQ(i))=\langle[\h], [\mathcal{Q}_1]\rangle.$
  \item Else, $\mathrm{Am}_s(\kk/\QQ(i))=\langle[\h], [\hh]\rangle.$
   \item  there  exists an unambiguous ideal
   $\mathcal{I}$ in $\kk/\QQ(i)$ of order  $2$ such that
   $$\mathrm{Am}(\kk/\QQ(i))=
    \left\{\begin{array}{ll}
     \langle[\h], [\mathcal{Q}_1], [\mathcal{I}]\rangle,& \text{ if }2p(x\pm1)\text{ is a square in }\NN,\\
     \langle[\h], [\hh], [\mathcal{I}]\rangle, & \text{ otherwise}.
\end{array}\right.$$
\end{enumerate}

The ideal $\mathcal{I}$ can be constructed by using the result:
     \begin{lem}[\cite{Si-95}]\label{3:115}
     Let $p_1$, $p_2$,...,$p_n$ be distinct primes and for each $j$,  let $e_j=\pm1$.
  Then there exist infinitely many primes $\ell$ such that $\left(\frac{p_j}{\ell}\right)=e_j$, for all $j$.
    \end{lem}
 Let  $\ell$ be a prime congruent to 1 $\pmod4$ and satisfying  $\left(\frac{pq_1q_2}{\ell}\right)=-\left(\frac{q_1}{\ell}\right)=1$, thus $\ell$ splits completely in  $\kk$. Therefore  $\mathcal{I}$ is one of the ideals of $\kk$  above  $\ell$; since   $\left(\frac{q_1}{\ell}\right)=-1$, so $\mathcal{I}$ remaind  inert in $\KK_2$. We proceed as in \cite{AZT14-3} to prove  that $\mathcal{I}$, $\h\mathcal{I}$,  $\hh\mathcal{I}$ and $\h\hh\mathcal{I}$ or $\mathcal{I}$, $\h\mathcal{I}$,  $\Q_1\mathcal{I}$ and $\Q_1\h\mathcal{I}$ are not principal in $\kk$.
\begin{The}\label{235}
Keep the previous hypothesis and notations and put  $\epsilon_{pq_2}=a+b\sqrt{pq_2}$, $\epsilon_{pq_1q_2}=x+y\sqrt{pq_1q_2}$.
\begin{enumerate}[\rm1.]
\item If   $a\pm1$ is a square in $\NN$ and   $2q_1(x+1)$, $2q_1(x-1)$ are not, then $\kappa_{\KK_2}=\langle[\Q_1], [\mathcal{I}]\rangle$ or $\langle[\Q_1], [\h\mathcal{I}]\rangle$.
\item If  $a\pm1$ and   $2q_1(x\pm1)$  are squares  in $\NN$, then $\kappa_{\KK_2}=\langle[\mathcal{I}]\rangle$ or $\langle[\mathcal{I}\h]\rangle$ or $\langle[\mathcal{I}\hh]\rangle$ or $\langle[\mathcal{I}\h\hh]\rangle$.
\item If $a+1$ and $a-1$ are not squares in $\NN$  and   $2q_1(x\pm1)$  is, then $\kappa_{\KK_2}=\langle[\Q_2]\rangle=\langle[\h\hh]\rangle$.
\item If  $a+1$, $a-1$,   $2q_1(x+1)$ and $2q_1(x-1)$ are not squares in $\NN$, then $\kappa_{\KK_2}=\langle[\Q_1]\rangle$.
\end{enumerate}
\end{The}
\begin{proof}
Let $\h$, $\hh$, $\Q_1$ and $\Q_2$  denote always the ideals of  $\kk$ above $\pi_1=e+2if$, $\pi_2=e-2if$, $q_1$  and $q_2$ respectively.\\
1.  Suppose  $a\pm1$ is a square in $\NN$ and  $2q_1(x+1)$, $2q_1(x-1)$ are not. We know  according to Proposition \ref{31} that  $E_{\KK_2}=\langle i, \sqrt{\epsilon_{q_1}\epsilon_{pq_2}},  \epsilon_{pq_1q_2}, \sqrt{i\epsilon_{q_1}}\rangle$ and that four  classes capitulate in $\KK_2$ one of them is  $\Q_1$. To proof the result, it suffices to  prove that  $\h$ does not capitulate in $\KK_2$.

If $\h$ capitulates in  $\KK_2$, then there exists $\alpha\in\KK_2$ such that $\h=(\alpha)$; hence   $(\alpha^2)= (\pi_1)$. As a result, there exists a unit   $\epsilon\in\KK_2$ such that $\pi_1\epsilon=\alpha^2$. The unit  $\epsilon$ can not be real or purely imaginary. In fact, if it is real   (same  proof if it is purely imaginary), then by putting  $\alpha=\alpha_1+i\alpha_2$, where $\alpha_i$ are in $\KK_2^+$, we get $\alpha_1^2-\alpha_2^2+2\alpha_1\alpha_2=\epsilon(e+2if)$, thus
$$\left\{
 \begin{array}{ll}
 \alpha_1^2-\alpha_2^2&=e\epsilon,\\
 \alpha_1\alpha_2&=f\epsilon,
 \end{array}\right.$$
  hence $f\alpha_1^2-e\alpha_2\alpha_1-f\alpha_2^2=0$.  But this implies that  $\alpha_1=\frac{\alpha_2(e\pm\sqrt p)}{f}$, and thus  $\sqrt p\in\KK_2^+$, which is absurd.

As $\pi_1\epsilon=\alpha^2$, so, by the norm $N_{\KK_2/\kk}$,  we get $\pi_1^2N_{\KK_2/\kk}(\epsilon)=N_{\KK_2/\kk}(\alpha)^2$ with $N_{\KK_2/\kk}(\epsilon)\in E_{\kk}=\langle i, \epsilon_{pq_1q_2}\rangle$. Therefore, we have the following result  $$N_{\KK_2/\kk}(\epsilon)\in\{\pm1, \pm i, \pm \epsilon_{pq_1q_2}, \pm i\epsilon_{pq_1q_2}\}.$$
\begin{enumerate}[\rm a.]
  \item If $N_{\KK_2/\kk}(\epsilon)=\pm i$, then $\pi_1^2(\pm i)=N_{\KK_2/\kk}(\alpha)^2$; hence $\sqrt i\in\kk$, which is absurd.
  \item If  $N_{\KK_2/\kk}(\epsilon)=\pm\epsilon_{pq_1q_2}$, then $\pi_1^2(\pm\epsilon_{pq_1q_2})=N_{\KK_2/\kk}(\alpha)^2$; this in turn yields that  $\sqrt{\epsilon_{pq_1q_2}}\in\kk$, which is absurd.
  \item    If $N_{\KK_2/\kk}(\epsilon)=\pm i\epsilon_{pq_1q_2}$, then $\pi_1^2(\pm i\epsilon_{pq_1q_2})=N_{\KK_2/\kk}(\alpha)^2$; this in turn yields that   $\sqrt{i\epsilon_{pq_1q_2}}\in\kk$, which is absurd.
    \item If  $N_{\KK_2/\kk}(\epsilon)=1$, then there exist $a$, $b$, $c$ and $d$ in $\{0, 1\}$ such that $\epsilon=i^a\sqrt{\epsilon_{q_1}\epsilon_{pq_2}}^b\epsilon_{pq_1q_2}^c \sqrt{i\epsilon_{q_1}}^d$ and $N_{\KK_2/\kk}(\epsilon)=1$, hence $(-1)^a\epsilon_{pq_1q_2}^{2c}i^d=1$. Thus  obviously we must have $a=c=d=0$. As a result, we get  $\epsilon=\sqrt{\epsilon_{q_1}\epsilon_{pq_2}}^b$ is a real, which is absurd.
      \item If $N_{\KK_2/\kk}(\epsilon)=-1$, then, by applying the same  argument, we get  $\epsilon=i\sqrt{\epsilon_{q_1}\epsilon_{pq_2}}^b$, which is purely imaginary, and this is absurd.
\end{enumerate}
To complete the proof of the first point of the corollary, we give examples that affirm the two cases of capitulation:
\begin{exams}~\\
$a\pm1$ is a square in $\NN$ and   $2q_1(x+1)$, $2q_1(x-1)$ are not.
\begin{longtable}{| c | c | c | c | c |  }
\hline
 $d$  $=p.q_1.q_2$
             & $\mathcal{I}$ in $\kk$ & $\h$ & $\mathcal{I}$ in $\KK_2$  & $\h\mathcal{I}$ in $\KK_2$\\
\hline
\endfirsthead
\hline
 $d$  $=p.q_1.q_2$
             & $\mathcal{I}$ in $\kk$ & $\h$ & $\mathcal{I}$ in $\KK_2$  & $\h\mathcal{I}$ in $\KK_2$\\
\hline
\endhead
$4029=17.3.79$ & $[5, 1, 1]~$ & $[170, 0]~$ & $[170, 0]~$ & $[0, 0]~$ \\ \hline
$4029=17.79.3$ & $[5, 0, 0]~$ & $[30, 0]~$ & $[30, 0]~$ & $[0, 0]~$ \\ \hline
$4029=17.79.3$ & $[0, 0, 1]~$ & $[30, 0]~$ & $[30, 0]~$ & $[0, 0]~$ \\ \hline
$4029=17.79.3$ & $[5, 0, 0]~$ & $[30, 0]~$ & $[30, 0]~$ & $[0, 0]~$ \\ \hline
$4029=17.79.3$ & $[5, 0, 0]~$ & $[30, 0]~$ & $[30, 0]~$ & $[0, 0]~$ \\ \hline
$4029=17.79.3$ & $[0, 1, 1]~$ & $[30, 0]~$ & $[30, 0]~$ & $[0, 0]~$ \\ \hline
$4029=17.79.3$ & $[0, 0, 1]~$ & $[30, 0]~$ & $[30, 0]~$ & $[0, 0]~$ \\ \hline
$4029=17.79.3$ & $[5, 1, 0]~$ & $[30, 0]~$ & $[30, 0]~$ & $[0, 0]~$ \\ \hline
\end{longtable}
\end{exams}
2. Suppose  $a\pm1$ and  $2q_1(x\pm1)$ are squares in $\NN$; then  according to  Proposition \ref{31}, $2p(x+1)$, $2p(x-1)$,  $2q_2(x+1)$ and $2q_2(x-1)$ are not squares in $\NN$; but  $2pq_2(x\pm1)$ is. Therefore   \cite[Proposition 2]{AZT12-2}  implies that   $\h\hh$ and  $\Q_2$ are  not principal in  $\kk$, but $\Q_1$ and $\h\hh\Q_2$ are; hence $[\Q_2]=[\h\hh]$. Which implies that   $\mathrm{Am}(\kk/\QQ(i))=\langle[\h], [\hh], [\mathcal{I}]\rangle$. By using the same method applied in the above point, we show that $\h$, $\hh$ and  $[\Q_2]=[\h\hh]$ do not capitulate in $\KK_2$. Thus $\kappa_{\KK_2}$ consists of one of the following ideal classes: $\mathcal{I}$, $\h\mathcal{I}$, $\hh\mathcal{I}$ and $\h\hh\mathcal{I}$. The following examples highlight these statements:
\begin{exams}~\\
$a\pm1$ and  $2q_1(x\pm1)$ are squares in $\NN$.
\begin{longtable}{| c | c | c | c | c |}
\hline
 $d$  $=p.q_1.q_2$  & $\mathcal{I}$ in $\KK_2$ & $\mathcal{I}\h$ in $\KK_2$ & $\mathcal{I}\hh$ in $\KK_2$ & $\h\hh\mathcal{I}$ in $\KK_2$\\
\hline
\endfirsthead
\hline
 $d$  $=p.q_1.q_2$  & $\mathcal{I}$ in $\KK_2$ & $\mathcal{I}\h$ in $\KK_2$ & $\mathcal{I}\hh$ in $\KK_2$ & $\h\hh\mathcal{I}$ in $\KK_2$\\
\hline
\endhead
$969=17.19.3$ & $[0, 0, 0, 0]~$ & $[3, 1, 1, 1]~$ & $[0, 1, 1, 0]~$ & $[0, 1, 0, 0]~$ \\ \hline
$1533=73.3.7$ & $[0, 0, 1, 0]~$ & $[0, 0, 0, 0]~$ & $[21, 1, 0, 1]~$ & $[21, 1, 0, 0]~$ \\ \hline
$2037 = 97.3.7$ & $[9, 0, 0, 0]~$ & $[9, 0, 0, 1]~$ & $[0, 0, 0, 0]~$ & $[0, 0, 0, 1]~$ \\ \hline
$2193 = 17.43.3$ & $[3, 0, 0, 0]~$ & $[0, 0, 1, 0]~$ & $[0, 1, 1, 0]~$ & $[0, 0, 0, 0]~$ \\ \hline
\end{longtable}
\end{exams}
3.  Suppose  $a+1$ and $a-1$ are not squares in $\NN$, and assume  $2q_1(x\pm1)$ is. Then  Propositions 1 and 2 of \cite{AZT12-2}  imply that $\Q_1$ is principal in $\kk$,  $\Q_2$ and $\h\hh$ are not, and   $[\Q_2]=[\h\hh]$. Moreover, $p(a\pm1)$ or $2p(a\pm1)$ is a square in $\NN$, hence $q_2\epsilon_{pq_2}$ or $2q_2\epsilon_{pq_2}$ is a square in $\KK_2$; and this yields that $\Q_2$ and  $\h\hh$ capitulate  in $\KK_2$. Here are some examples that illustrate our results.
\begin{exams}~\\
$a+1$ and $a-1$ are not squares in $\NN$ and  $2q_1(x\pm1)$ is.
\footnotesize
\begin{longtable}{| c | c | c | c | c | c | c | c |}
\hline
 $d = p.q_1.q_2$  & $a$ & $2q_1(x+1)$ & $2q_1(x-1)$ & $\h\hh$   & $\Q_2$  & $\h\hh$ & $\Q_2$ \\
                 &    &                 &             &in $\kk$ & in $\kk$ & &\\
\hline
\endfirsthead
\hline
 $d = p.q_1.q_2$  & $a$ & $2q_1(x+1)$ & $2q_1(x-1)$ & $\h\hh$   & $\Q_2$  & $\h\hh$ & $\Q_2$ \\
                 &    &                 &             &in $\kk$ & in $\kk$ & &\\
\hline
\endhead
$165=5.11.3$ & $4$ & $165$ & $121=11^2$ & $[2, 0]~$ & $[2, 0]~$ & $[0, 0, 0]~$ & $[0, 0, 0]~$ \\ \hline
$273=13.3.7$ & $1574$ & $4368$ & $4356=66^2$ & $[2, 0]~$ & $[2, 0]~$ & $[0, 0, 0]~$ & $[0, 0, 0]~$ \\ \hline
$285=5.19.3$ & $4$ & $361=19^2$ & $285$ & $[4, 0]~$ & $[4, 0]~$ & $[0, 0, 0]~$ & $[0, 0, 0]~$ \\ \hline
$385=5.11.7$ & $6$ & $2108304=1452^2$ & $2108260$ & $[2, 0]~$ & $[2, 0]~$ & $[0, 0, 0]~$ & $[0, 0, 0]~$ \\ \hline
$429=13.3.11$ & $12$ & $441=21^2$ & $429$ & $[4, 0]~$ & $[4, 0]~$ & $[0, 0, 0]~$ & $[0, 0, 0]~$ \\ \hline
$465=5.31.3$ & $4$ & $984064=992^2$ & $983940$ & $[4, 0]~$ & $[4, 0]~$ & $[0, 0, 0]~$ & $[0, 0, 0]~$ \\ \hline
$609=29.7.3$ & $28$ & $8479744=2912^2$ & $8479716$ & $[4, 0]~$ & $[4, 0]~$ & $[0, 0, 0]~$ & $[0, 0, 0]~$ \\ \hline
$665=5.19.7$ & $6$ & $521360$ & $521284=722^2$ & $[6, 0]~$ & $[6, 0]~$ & $[0, 0, 0]~$ & $[0, 0, 0]~$ \\ \hline
$741=13.3.19$ & $85292$ & $741$ & $729=27^2$ & $[6, 0]~$ & $[6, 0]~$ & $[0, 0, 0]~$ & $[0, 0, 0]~$ \\ \hline
$777=37.7.3$ & $295$ & $3136=56^2$ & $3108$ & $[4, 0]~$ & $[4, 0]~$ & $[0, 0, 0]~$ & $[0, 0, 0]~$ \\ \hline
$885=5.59.3$ & $4$ & $14160$ & $13924=118^2$ & $[6, 0]~$ & $[6, 0]~$ & $[0, 0, 0]~$ & $[0, 0, 0]~$ \\ \hline
$897=13.3.23$ & $415$ & $3600=60^2$ & $3588$ & $[4, 0]~$ & $[4, 0]~$ & $[0, 0, 0]~$ & $[0, 0, 0]~$ \\ \hline
$1045=5.11.19$ & $39$ & $1089=33^2$ & $1045$ & $[4, 0]~$ & $[4, 0]~$ & $[0, 0, 0]~$ & $[0, 0, 0]~$ \\ \hline
\end{longtable}
\normalsize
\end{exams}
4. If  $a+1$,  $a-1$,   $2q_1(x+1)$ and $2q_1(x-1)$ are not squares in $\NN$, then $\Q_1$ is not  principal in $\kk$; and as $\sqrt{q_1}\in \KK_2$, so $\Q_1$ capitulate in $\KK_2$.
\begin{exams}~\\
$a+1$,  $a-1$,   $2q_1(x+1)$ and  $2q_1(x-1)$ are not squares in $\NN$.
\begin{longtable}{| c | c | c | c | c | c | c |}
\hline
 $d = p.q_1.q_2$  & $a+1$ & $a-1$ & $2q_1(x+1)$ & $2q_1(x-1)$  & $\Q_1$ in $\kk$ & $\Q_1$\\
\hline
\endfirsthead
\hline
 $d = p.q_1.q_2$  & $a+1$ & $a-1$ & $2q_1(x+1)$ & $2q_1(x-1)$  & $\Q_1$ in $\kk$ & $\Q_1$\\
\hline
\endhead
$105=5.7.3$ & $5$ & $3$ & $588$ & $560$ & $[2, 0]~$ & $[0, 0]~$ \\ \hline
$165=5.3.11$ & $90$ & $88$ & $45$ & $33$ & $[2, 0]~$ & $[0, 0]~$ \\ \hline
$273=13.7.3$ & $26$ & $24$ & $10192$ & $10164$ & $[2, 0]~$ & $[0, 0]~$ \\ \hline
$285=5.3.19$ & $40$ & $38$ & $57$ & $45$ & $[4, 0]~$ & $[0, 0]~$ \\ \hline
$345=5.3.23$ & $1127$ & $1125$ & $40572$ & $40560$ & $[2, 0]~$ & $[0, 0]~$ \\ \hline
$345=5.23.3$ & $5$ & $3$ & $311052$ & $310960$ & $[2, 0]~$ & $[0, 0]~$ \\ \hline
$385=5.7.11$ & $90$ & $88$ & $1341648$ & $1341620$ & $[2, 0]~$ & $[0, 0]~$ \\ \hline
$429=13.11.3$ & $26$ & $24$ & $1617$ & $1573$ & $[4, 0]~$ & $[0, 0]~$ \\ \hline
$465=5.3.31$ & $250$ & $248$ & $95232$ & $95220$ & $[4, 0]~$ & $[0, 0]~$ \\ \hline
\end{longtable}
\end{exams}
\end{proof}
\subsection{Capitulation in $\KK_3$}~\\
Let  $p$,  $q_1$ and $q_2$ be different primes satisfying  $p\equiv-q_1\equiv-q_2\equiv1 \pmod 4$. Put $\kk=\QQ(\sqrt{pq_1q_2}, i)$,  $\KK_3=\QQ(\sqrt{q_2},\sqrt{pq_1}, i)$ and $\epsilon_{pq_1}=a+b\sqrt{pq_1}$. As  $q_1$ and $q_2$ play symmetric roles, so the following results are deduced from the above  by analogy. Let $\mathcal{I}$ be the ideal defined as above and assume   the prime $\ell$  satisfies the conditions: $\ell\equiv1 \pmod4$ and  $\left(\frac{pq_1q_2}{\ell}\right)=-\left(\frac{q_2}{\ell}\right)=1$.
\begin{The}\label{237}
Keep the obvious notations and hypothesis. Put $\epsilon_{pq_1}=a+b\sqrt{pq_1}$, then
\begin{enumerate}[\rm1.]
\item If  $a\pm1$ is a square in $\NN$ and  $2q_2(x+1)$, $2q_2(x-1)$ are not, then $\kappa_{\KK_3}=\langle[\Q_2], [\mathcal{I}]\rangle$ or $\langle[\Q_2], [\h\mathcal{I}]\rangle$.
\item If $a\pm1$ and   $2q_2(x\pm1)$  are squares in $\NN$, then $\kappa_{\KK_3}=\langle[\mathcal{I}]\rangle$ or $\langle[\mathcal{I}\h]\rangle$ or $\langle[\mathcal{I}\hh]\rangle$ or $\langle[\mathcal{I}\h\hh]\rangle$.
\item If $a+1$ and $a-1$ are not squares  in $\NN$ and   $2q_2(x\pm1)$  is, then  $\kappa_{\KK_3}=\langle[\Q_1]\rangle$.
\item If  $a+1$,  $a-1$,  $2q_2(x+1)$ and $2q_2(x-1)$  are not squares in $\NN$, then  $\kappa_{\KK_3}=\langle[\Q_2]\rangle$.
\end{enumerate}
\end{The}
\subsection{Capitulation in $\k$}~\\
The following theorem is a simple deduction from Theorems \ref{233}, \ref{235} and \ref{237}.
\begin{The}
Let  $p$,  $q_1$ and $q_2$ be different primes satisfying  $p\equiv-q_1\equiv-q_2\equiv1 \pmod 4$. Put $\kk=\QQ(\sqrt{pq_1q_2}, i)$ and denote by $\k$ its genus field. Let $\epsilon_{pq_1q_2}=x+y\sqrt{pq_1q_2}$ be the fundamental unit of   $\QQ(\sqrt{pq_1q_2})$.
 \begin{enumerate}[\rm1.]
 \item Assume $p\equiv1\pmod8$, then there exists an unambiguous ideal  $\mathcal{I}$ of  $\kk/\QQ(i)$ of order $2$ such that:
\begin{enumerate}[\rm i.]
   \item If  $2p(x\pm1)$ is a square in $\NN$, then $\langle[\h], [\Q_1],  [\mathcal{I}]\rangle\subseteq \kappa_{\G}$.
  \item Else,  $\langle[\h], [\hh],  [\mathcal{I}]\rangle\subseteq \kappa_{\G}$.
  \end{enumerate}
\item Assume $p\equiv5\pmod8$.
  \begin{enumerate}[\rm i.]
   \item If  $2p(x\pm1)$ is a square in $\NN$, then $\langle[\h], [\Q_1]\rangle\subseteq \kappa_{\G}$.
  \item Else,  $\langle[\h], [\hh]\rangle\subseteq \kappa_{\G}$.
  \end{enumerate}
\end{enumerate}
\end{The}
\section{\bf{Application}}
Let $p\equiv-q_1\equiv-q_2\equiv1 \pmod 4$ be different primes such that  $\mathbf{C}l_2(\mathds{\kk})$ is of type $(2, 2, 2)$.
According to  \cite{AzTa-08}, $\mathbf{C}l_2(\kk)$  is of type $(2, 2, 2)$ if and only if  $p$, $q_1$ and $q_2$  satisfy the   following two conditions:
\begin{enumerate}[\rm A:]
 \item      $p\equiv -q_1\equiv -q_2\equiv1 \pmod4$ and  $\left(\frac{2}{p}\right)=\left(\frac{q_1}{q_2}\right)=-\left(\frac{q_2}{q_1}\right)=1$.
\item     One of the following three conditions is satisfied:
 \begin{enumerate}[\upshape\indent(I):]
   \item  $\left(\frac{p}{q_1}\right)\left(\frac{p}{q_2}\right)=-1$ and $\left(\frac{2}{q_1}\right)=\left(\frac{2}{q_2}\right)=-1$.
   \item $\left(\frac{p}{q_1}\right)\left(\frac{p}{q_2}\right)=-1$,  $\left(\frac{2}{q_1}\right)=1$ and $\left(\frac{2}{q_2}\right)=-1$.
   \item $\left(\frac{p}{q_1}\right)=\left(\frac{p}{q_2}\right)=-1$ and $\left(\frac{2}{q_1}\right)\left(\frac{2}{q_2}\right)=-1$.
 \end{enumerate}
 \end{enumerate}
 \begin{rema}
We  keep the notations defined in \cite[Definition 1]{ZAT-15}, and we add the following definition assuming  $p\equiv -q_1\equiv -q_2\equiv1 \pmod4$ satisfying the condition A.
\begin{enumerate}[\rm1.]
  \item  $p$, $q_1$ and $q_2$ are said of  type $B(III)(1)$ if    $\left(\frac{p}{q_1}\right)=\left(\frac{p}{q_2}\right)=-1$ and $-\left(\frac{2}{q_1}\right)=\left(\frac{2}{q_2}\right)=1$.
  \item $p$, $q_1$ and $q_2$ are said of  type $B(III)(2)$ if    $\left(\frac{p}{q_1}\right)=\left(\frac{p}{q_2}\right)=-1$ and $\left(\frac{2}{q_1}\right)=-\left(\frac{2}{q_2}\right)=1$.
\end{enumerate}
\end{rema}
To continue we need the following results.
\begin{lem}\label{11}
Let $p\equiv-q_1\equiv-q_2\equiv1 \pmod 4$ be different primes satisfying the condition A, and put $\epsilon_{pq_2}=a+b\sqrt{pq_2}$.
\begin{enumerate}[\rm1.]
  \item If  $p$, $q_1$ and $q_2$ are of  type $B(I)$ or  $B(II)$,  then  $a+1$ is not a  square in $\NN$.
    \item If  $p$, $q_1$ and $q_2$ are of  type $B(I)(1)$ or  $B(II)(1)$,  then  $p(a-1)$ and $2p(a+1)$ are not  squares in $\NN$.
  \item If  $p$, $q_1$ and $q_2$ are of  type $B(I)(2)$ or  $B(II)(2)$,  then  $p(a+1)$ and $2p(a-1)$ are not  squares in $\NN$.
   \item If  $p$, $q_1$ and $q_2$ are of  type $B(III)(1)$,  then  $a-1$ and $p(a+1)$ are  not  squares in $\NN$.
  \item If  $p$, $q_1$ and $q_2$ are of  type  $B(III)(2)$,  then  $a+1$ and $p(a-1)$ are not   squares in $\NN$.
  \item If  $p$, $q_1$ and $q_2$ are of  type  $B(III)$,  then  $2p(a+1)$ is not a square in $\NN$.
\end{enumerate}
\end{lem}
\begin{proof}
We know that  $N(\epsilon_{pq_2})=1$, then $a^2-1=b^2pq_2$, hence by  Lemma \ref{1:046} and the decomposition uniqueness in $\ZZ$  there exist $b_1$, $b_2$ in $\ZZ$ such that:
\begin{center}
   (1) $\left\{\begin{array}{ll}
 a\pm1=b_1^2,\\
 a\mp1=pq_2b_2^2;
 \end{array}\right.$ or (2)
    $\left\{\begin{array}{ll}
 a\pm1=pb_1^2,\\
 a\mp1=q_2b_2^2;
 \end{array}\right.$ or (3)
    $\left\{\begin{array}{ll}
 a\pm1=2pb_1^2,\\
 a\mp1=2q_2b_2^2;
 \end{array}\right.$
 \end{center}
1.  Suppose $$\left\{\begin{array}{ll}
 a+1=b_1^2,\\
 a-1=pq_2b_2^2,
 \end{array}\right.$$
 then $\left(\frac{2}{q_2}\right)=1$, but this contradicts the conditions  $B(I)$ and  $B(II)$, hence the result.

The other cases are checked similarly.
\end{proof}
\begin{rema}
If $\mathbf{C}l_2(\mathds{\kk})$ is of type $(2, 2, 2)$, then
by Proposition \ref{248} and  \cite[Lemma 3]{ZAT-15}, we deduce that:
\begin{enumerate}[\rm1.]
  \item $\mathrm{Am}_s(\kk/\QQ(i))=\langle[\h], [\mathcal{Q}_1]\rangle\varsubsetneq\mathrm{Am}(\kk/\QQ(i))=\mathbf{C}l_2(\kk)=\langle[\h], [\mathcal{Q}_1], [\mathcal{I}]\rangle$,  if $p$, $q_1$ and $q_2$ are of type $B(III)$,
  \item  $\mathrm{Am}_s(\kk/\QQ(i))=\langle[\h], [\hh]\rangle\varsubsetneq\mathrm{Am}(\kk/\QQ(i))=\mathbf{C}l_2(\kk)=\langle[\mathcal{H}_1], [\mathcal{H}_2], [\mathcal{I}]\rangle$, otherwise.
\end{enumerate}
\end{rema}
\begin{The}\label{7:009}
Let $p\equiv-q_1\equiv-q_2\equiv1 \pmod 4$ be different primes such that  $\mathbf{C}l_2(\mathds{\kk})$ is of type $(2, 2, 2)$, where
  $\kk=\QQ(\sqrt{pq_1q_2}, i)$.
  \begin{enumerate}[\rm1.]
\item Exactly four classes  of $\mathbf{C}l_2(\kk)$ capitulate in  $\KK_1$.
\begin{enumerate}[\rm i.]
 \item If $p$,  $q_1$ and $q_2$ are of type $B(III)$, then $\kappa_{\KK_1}=\langle[\h], [\mathcal{Q}_1]\rangle$.
 \item Else,  $\kappa_{\KK_1}=\langle[\h], [\hh]\rangle$.
 \end{enumerate}
 \item Put $\epsilon_{pq_2}=a+b\sqrt{pq_2}$, then the capitulation in $\KK_2$ is given by:
\begin{enumerate}[\rm i.]
\item If  $p$,  $q_1$ and $q_2$ are of type $B(I)(1)$ or $B(II)(1)$, then $\kappa_{\KK_2}=\langle[\mathcal{I}]\rangle$ or $\langle[\h\mathcal{I}]\rangle$ or $\langle[\hh\mathcal{I}]\rangle$ or $\langle[\h\hh\mathcal{I}]\rangle$.
\item If  $p$,  $q_1$ and $q_2$ are of type $B(I)(2)$ or $B(II)(2)$, then
\begin{enumerate}[\rm a.]
\item If  $a-1$ is a square in $\NN$, then $\kappa_{\KK_2}=\langle[\h\hh], [\mathcal{I}]\rangle$ or
 $\langle[\h\hh], [\h\mathcal{I}]\rangle$.
\item Else,   $\kappa_{\KK_2}=\langle[\Q_1]\rangle=\langle[\h\hh]\rangle$.
\end{enumerate}
\item If $p$,  $q_1$ and $q_2$ are of type  $B(III)$, then
 \begin{enumerate}[\rm a.]
\item If  $a\pm1$ is a square in $\NN$, then  $\kappa_{\KK_2}=\langle[\Q_1], [\mathcal{I}]\rangle$ or $\langle[\Q_1], [\h\mathcal{I}]\rangle$.
\item Else,   $\kappa_{\KK_2}=\langle[\Q_1]\rangle$.
\end{enumerate}
\end{enumerate}
\item Put $\epsilon_{pq_1}=a+b\sqrt{pq_1}$, then the capitulation in $\KK_3$ is given by:
\begin{enumerate}[\rm i.]
\item If  $p$,  $q_1$ and $q_2$ are of type $B(I)(2)$ ou $B(II)(2)$, then $\kappa_{\KK_3}=\langle[\mathcal{I}]\rangle$ or $\langle[\h\mathcal{I}]\rangle$ or $\langle[\hh\mathcal{I}]\rangle$ or $\langle[\h\hh\mathcal{I}]\rangle$.
\item If  $p$,  $q_1$ and $q_2$ are of type $B(I)(1)$ or $B(II)(1)$, then
\begin{enumerate}[\rm a.]
\item If  $a-1$ is a square in $\NN$, then $\kappa_{\KK_3}=\langle[\h\hh], [\mathcal{I}]\rangle$ or $\langle[\h\hh], [\h\mathcal{I}]\rangle$.
\item Else,   $\kappa_{\KK_3}=\langle[\Q_2]\rangle=\langle[\h\hh]\rangle$.
\end{enumerate}
\item If  $p$,  $q_1$ and $q_2$ are of type  $B(III)$, then
 \begin{enumerate}[\rm a.]
\item If  $a\pm1$ is a square in $\NN$, then  $\kappa_{\KK_3}=\langle[\Q_2], [\mathcal{I}]\rangle$ or $\langle[\Q_2], [\h\mathcal{I}]\rangle$.
\item Else,   $\kappa_{\KK_2}=\langle[\Q_2]\rangle$.
\end{enumerate}
\end{enumerate}
 \end{enumerate}
\end{The}
\begin{proof} Let $\epsilon_{pq_1q_2}=x+y\sqrt{pq_1q_2}$ denote the fundamental unit of $\QQ(\sqrt{pq_1q_2})$.
\begin{enumerate}[\rm1.]
 \item We know, by \cite[Lemma 3]{ZAT-15}, that if  $p$,  $q_1$ and $q_2$ are of  type  $B(III)$, then $2p(x-1)$ is a square in $\NN$, and otherwise  $2p(x-1)$, $2p(x+1)$ are not squares in $\NN$. Thus   Theorem \ref{233}  implies  the results.
\item Put $\epsilon_{pq_2}=a+b\sqrt{pq_2}$.
 \begin{enumerate}[\rm i.]
 \item Suppose  $p$,  $q_1$ and $q_2$ satisfy the conditions  $A$ and  $B(I)(1)$ or $B(II)(1)$, then, by \cite[Lemma 3]{ZAT-15},  $2q_1(x+1)$ is a square in $\NN$. On the other hand, from  Lemma \ref{11},  $p(a-1)$ and  $2p(a+1)$ are not squares in  $\NN$, thus  $a-1$ is a square in $\NN$. Therefore, we are in the hypotheses of Theorem \ref{235}(2), thus the results.
 \item  Suppose  $p$,  $q_1$ and $q_2$ satisfy the conditions $A$ and  $B(I)(2)$ or $B(II)(2)$, then, by \cite[Lemma 3]{ZAT-15}, $2q_2(x-1)$ is a square in $\NN$ i.e. $2pq_1(x+1)$ is a square in $\NN$. Thus  \cite[Proposition 1]{AZT12-2} implies that $[\h\hh]=[Q_1]$.
 On the other hand, from  Lemma \ref{11}, one of the numbers  $a-1$,  $p(a-1)$ or   $2p(a+1)$ is  a square in  $\NN$. So we are in the hypotheses of Theorem \ref{235} (1) or (4), thus the results.
\item  Suppose  $p$,  $q_1$ and $q_2$ satisfy the conditions $A$ and  $B(III)$, then, by \cite[Lemma 3]{ZAT-15}, $2p(x-1)$ is a square in $\NN$, and by Lemma \ref{11}, $2p(a+1)$ is not a square in $\NN$.\\
    $\bullet$ If $p$,  $q_1$ and $q_2$ are of type   $B(III)(1)$, then Lemma \ref{11} implies that one of the numbers $a+1$, $p(a-1)$ or $2p(a-1)$ is a square in $\NN$.\\
    $\bullet$ If $p$,  $q_1$ and $q_2$ are of type   $B(III)(2)$, then Lemma \ref{11} implies that one of the numbers $a-1$, $p(a+1)$ or $2p(a-1)$ is a square in $\NN$.\\
    Therefore,\\
    a. If $a\pm1$ is a square in $\NN$, then the result is assured by Theorem \ref{235}(1).\\
    b. Else, the result is assured by Theorem \ref{235}(4).
  \end{enumerate}
 \item These results are shown as in 2.
\end{enumerate}
\end{proof}
\begin{coro}\label{18}
Keep the hypotheses and notations mentioned in  Theorem $\ref{7:009}$. Then all the classes of   $\mathbf{C}l_2(\kk)$ capitulate in $\G$ i.e.
$$\kappa_{\G}=\mathbf{C}l_2(\kk)=\mathrm{Am}(\kk/\QQ(i)).$$
\end{coro}
\section{Acknowledgement}
We would like to thank the referee of our paper for his   precious remarks and suggestions.

\end{document}